\documentclass[reqno]{amsart}
\usepackage{amsmath}
\usepackage{amsthm}
\usepackage{amssymb}
\usepackage{stmaryrd}
\usepackage{afterpage}
\usepackage{rotating}    
\usepackage{adjustbox}   
\usepackage{booktabs}
\usepackage[utf8]{inputenc}
\usepackage[T1]{fontenc}
\usepackage{lmodern}
\usepackage{eucal}
\pdfoutput=1 
\usepackage{microtype}
\usepackage{url}
\usepackage{enumitem}
\usepackage{tikz}
\usetikzlibrary{arrows.meta}
\usepackage{tikz-cd}
\usepackage{bbm}
\usepackage{longtable,booktabs}
\usepackage[all]{xy}
\usepackage{leftindex}
\usepackage{float}
\usepackage{caption}
\usepackage{relsize}

\usepackage[letterpaper, left = 1.5in, right = 1.5in, top = 1.5in, bottom = 1.5in]{geometry}

\usepackage{pdflscape}

\usepackage{hyperref}

\usepackage{xcolor}
\hypersetup{
    colorlinks,
    linkcolor={red!50!black},
    citecolor={blue!50!black},
    urlcolor={blue!80!black}
}
\definecolor{chartgray}{gray}{0.4}
\definecolor{darkcyan}{rgb}{0, 0.7, 0.7}
\definecolor{darkblue}{rgb}{0, 0.65, 0}
\definecolor{truemagenta}{rgb}{1, 0, 1}

\usepackage{graphicx}

\usepackage{lipsum}     
\usepackage{booktabs} 
\usepackage{rotating}
\usepackage{tabularx}

\usepackage[colorinlistoftodos]{todonotes}

\usepackage[capitalise,nameinlink]{cleveref}
\crefname{section}{Section}{Sections}
\crefname{subsection}{Subsection}{Subsections}

\theoremstyle{definition}
\newtheorem{theorem}{Theorem}[section]

\newtheorem{prop}[theorem]{Proposition}

\crefname{prop}{Proposition}{Propositions}
\Crefname{prop}{Proposition}{Propositions}

\newtheorem{lemma}[theorem]{Lemma}

\crefname{lemma}{Lemma}{Lemmas}
\Crefname{lemma}{Lemma}{Lemmas}

\newtheorem{cor}[theorem]{Corollary}

\crefname{cor}{Corollary}{Corollaries}
\Crefname{cor}{Corollary}{Corollaries}

\crefname{ex}{Example}{Examples}
\Crefname{ex}{Example}{Examples}

\crefname{construction}{Construction}{Constructions}
\Crefname{construction}{Construction}{Constructions}

\newtheorem{rmk}[theorem]{Remark}

\newtheorem{construction}[theorem]{Construction}

\newtheorem*{rmk*}{Remark}
\newtheorem*{ex*}{Example}
\newtheorem*{theorem*}{Theorem}
\newtheorem*{defn*}{Definition}
\newtheorem*{question*}{Question}

\crefformat{questionmanual}{#2Question#3}
\Crefformat{questionmanual}{#2Question#3}

\newcommand{\bfh}{\mathbf{h}}

\newcommand{\bbZ}{\mathbb{Z}}

\newcommand{\bbF}{\mathbb{F}}

\newcommand{\bbR}{\mathbb{R}}

\newcommand{\calA}{\mathcal{A}}

\newcommand{\Res}{\operatorname{Res}}

\newcommand{\Ext}{\operatorname{Ext}}
\newcommand{\Cof}{\operatorname{Cof}}

\newcommand{\im}{\operatorname{im}}

\newcommand{\id}{\operatorname{id}}

\newcommand\xqed[1]{%
  \leavevmode\unskip\penalty9999 \hbox{}\nobreak\hfill
  \quad\hbox{#1}}
\newcommand\tqed{\xqed{$\triangleleft$}}

\makeatletter
\DeclareRobustCommand{\tvdots}{%
  \vbox{\baselineskip4\p@\lineskiplimit\bbZ@\kern0\p@\hbox{.}\hbox{.}\hbox{.}}}
\makeatother

\makeatletter
\@namedef{subjclassname@2020}{\textup{2020} Mathematics Subject Classification}
\makeatother

\makeatletter
\newcommand{\raisemath}[1]{\mathpalette{\raisem@th{#1}}}
\newcommand{\raisem@th}[3]{\raisebox{#1}{$#2#3$}}
\makeatother

\begin{document}

\title[The $RO(C_2)$-graded Hurewicz image of Atiyah's Real K-theory]{The $RO(C_2)$-graded Hurewicz image of Atiyah's Real K-theory}

\author[Sihao Ma]{Sihao Ma}
\address{Department of Mathematics, UCLA, Los Angeles, CA 90095-1555, USA}
\email{masihao@math.ucla.edu}

\author{XiaoLin Danny Shi}
\address{Department of Mathematics, University of Washington, Seattle, WA 98195, USA}
\email{dannyshi@uw.edu}

\author{Zhouli Xu}
\address{Department of Mathematics, UCLA, Los Angeles, CA 90095-1555, USA}
\email{xuzhouli@ucla.edu}

\author{Albert Jinghui Yang}
\address{Department of Mathematics, University of Pennsylvania, Philadelphia, PA 19104, USA}
\email{yangjh@sas.upenn.edu}

\author{Shangjie Zhang}
\address{Department of Mathematics, University of Washington, Seattle, WA 98195, USA}
\email{sjzh@uw.edu}
\subjclass[2020]{
19L47, 
55N91, 
55Q91. 
}

\begin{abstract}
  We determine the $RO(C_2)$-graded $k\bbR$-Hurewicz image, where $k \bbR$ denotes the connective Atiyah's Real $K$-theory spectrum. As a consequence, we obtain four infinite families in the $RO(C_2)$-graded equivariant stable stems that are $C_2$-equivariant refinements of the classical $\mu$-families. Using a $C_2$-equivariant $v_{1,0}$-self map, we explicitly construct an equivariant Adams periodicity operator $P_{C_2}$ that captures these $v_1$-periodicity phenomena in the $C_2$-stable stems.
\end{abstract}

\maketitle
\tableofcontents

\section{Introduction}

In their landmark paper \cite{hillhopkinsravenel2016nonexistence}, Hill, Hopkins, and Ravenel resolved the longstanding Kervaire invariant problem by showing that the elements $\theta_j$ do not exist for $j\geq 7$. A central ingredient in their proof is the construction and analysis of the norm spectrum  $N_{C_2}^{G}MU_{\bbR}$ for a cyclic 2-group $G$.  Here, $MU_{\bbR}$ is the $C_2$-equivariant \emph{Real bordism spectrum}, which originated in the work of Landweber, Fujii, and Araki \cite{landweber1968conjugations, fujii1975cobordism, araki1979orientations}.  When $G=C_2$, $MU_\bbR$ and its 2-local summand $BP_\bbR$ were analyzed by Hu and Kriz, who obtained structural and computational results for these spectra \cite{hukriz2001real}.

From the perspective of chromatic homotopy theory, the height-one truncation of the $2$-local Real Brown--Peterson spectrum $BP_{\bbR} \langle 1 \rangle$ is equivalent to the $C_2$-equivariant spectrum $k\bbR$, the connective cover of \emph{Atiyah's Real $K$-theory} $K\bbR$ that classifies Real vector bundles \cite{atiyah1966ktheory}. These two spectra provide a natural bridge between real and complex $K$-theories and a conceptual framework for the study of real elliptic operators and real algebraic geometry \cite{atiyah1966ktheory, dugger2005atiyah, greenleesmeier2017gorenstein}. More precisely,
\[k\bbR^{C_2} \simeq  ko,\quad  \Res^{C_2}_e(k\bbR) \simeq ku,\]
and $k\bbR$ and $K\bbR$ are related by inverting the height-one periodicity class $\overline{v}_1$. In particular, $K \bbR$ exhibits $RO(C_2)$-graded periodicities that refine both the complex and the real Bott periodicities. As such, they capture height-one chromatic information in the $C_2$-equivariant stable homotopy category and serve as basic examples in the study of $C_2$-equivariant chromatic phenomena. For recent work in this direction, see \cite{barthelhausmannnaumannnikolausnoelstapleton2019balmer, kuhnlloyd2022chromatic, lilormanquigley2022tate, behrenscarlisle2025periodic, lipeteresentatum2026spectrum}.

Real-oriented spectra are also computationally tractable objects for detecting classes in the $C_2$-equivariant sphere. More generally, for a ring spectrum $E$, the classical \emph{Hurewicz-image problem} asks for the image of the unit map 
\[
    \pi_*S^0\longrightarrow \pi_*E.
\]
The idea of detecting stable homotopy classes by computationally accessible homology theories is central to many developments in stable homotopy theory, such as the Detection Theorem of Hill--Hopkins--Ravenel \cite{hillhopkinsravenel2016nonexistence}. Other notable work includes Hurewicz-image computations for topological modular forms \cite{behrenshillhopkinsmahowald2020detecting,  brunerrognes2021adams, behrensmahowaldquigley2023hurewicz, bhattacharya2024new}, and for the integer-graded Real bordism theory and Real Johnson--Wilson theories \cite{lishiwangxu2019hurewicz, hahnshi2020real}. 

In recent joint work with Guo, Li, Lu, and Wu \cite{gllmwxyz2026hurewicz}, we extended this line of investigation to the $C_2$-equivariant setting by completely determining the $RO(C_2)$-graded Hurewicz images of the Eilenberg--MacLane spectra $H\underline{\mathbb F_2}$, $H\underline{\mathbb Z}$ and $H\underline{A}$. Here $\underline{\bbF_2}$ and $\underline{\bbZ}$ denote the constant Mackey functors with values in $\bbF_2$ and $\bbZ$, respectively, and $\underline{A}$ denotes the Burnside Mackey functor. Classically, the $H\bbF_2$- and $H\bbZ$-Hurewicz images are simply generated by the unit and concentrated in degree zero. It turns out that, equivariantly, $H\underline{\bbF_2}$ and $H\underline{\bbZ}$ detect infinitely many nontrivial classes at various $RO(C_2)$-degrees, and the pattern is closely tied to the classical problem of vector fields on spheres in differential topology \cite{adams1962vector}.

In this paper, we extend the study of the $RO(C_2)$-graded equivariant Hurewicz-image problem initiated in \cite{gllmwxyz2026hurewicz}: we determine the full $RO(C_2)$-graded Hurewicz image of Atiyah's Real $K$-theory. We focus on $k\bbR$, which has the advantage of admitting a non-nilpotent Euler class $\rho$ and an infinite $\rho$-divisible pattern in the negative cone, while retaining the $\bar{v}_1$-periodic pattern of $K\bbR$. In particular, the Hurewicz image of $K\bbR$ can be read off from our computations below.

Recall that the real representation ring of $C_2$ is $RO(C_2) \cong \bbZ[\sigma]/(\sigma^2-1)$, where $\sigma$ denotes the sign representation. We use the bidegree $(s, w)$ to indicate the $RO(C_2)$-degree $(s-w)+w\sigma$. We refer to $s$ as the \textit{stem}, $w$ as the \textit{weight}, and $s-w$ as the \textit{coweight}. Let $S^{0,0}$ denote the $C_2$-equivariant sphere spectrum, and let $S^0$ denote the non-equivariant sphere spectrum.

\begin{theorem}[\cref{thm:krhurewicz}] \label{thm:mainhurewicz}
    The $RO(C_2)$-graded image of the Hurewicz map \[\iota_{k\bbR}: S^{0,0} \longrightarrow k\bbR\] consists of
    \[\im (\iota_{k\bbR} )_{s, w}=\begin{cases}
        \bbF_2 \quad &\mathrm{if}\ s=w<0,\\
        \bbZ \quad &\mathrm{if}\ s=0\ \mathrm{and}\ w=2k\ \mathrm{for}\ k \in \bbZ,\\

        \bbF_2 \quad &\mathrm{if}\ s=w=8k+6+\epsilon, \ \ \mathrm{for}\ k\geq 0, \epsilon =0, 1,\\
        
        \bbF_2& \mathrm{if}\ s=8l+\epsilon\ \mathrm{and}\ s-w=4k+1 \ \ \mathrm{for}\ k\in \bbZ, l\geq 0, \epsilon =0, 1, \\
        
        \bbF_2 & \mathrm{if}\ s=8l+2 \ \mathrm{and}\ s-w=4k+2\ \  \mathrm{for}\ k\in \bbZ, l\geq 0, \\
       
        \bbF_2 \quad &\mathrm{if}\ 0\leq s< \psi(w-s-1) \ \mathrm{and}\ s-w=4k-1\ \  \mathrm{for\ } k\leq -1,\\
        0\quad &\mathrm{otherwise. \rlap{\hspace{18.5em} \tqed}}
    \end{cases}\]
\end{theorem}

\begin{figure}[p]
        \centering
        \begin{tikzpicture}[scale=0.7135]
\draw[thick,dashed, ->] (-6,0) -- (10.5,0);
\draw[thick,dashed, ->] (0,-10) -- (0, 16);

\node at (-5, 12){$\bullet= \bbF_2$};
\node at (-5, 11){$\square$ $= \bbZ$};
\draw (-5.7, 12.8)--(-5.2, 13.3);
\node at (-4.5, 13){$=\cdot \rho$};
\draw [red](-5.7, 13.9)--(-5.1, 14.2);
\node at (-4.5, 14){$=\cdot \bar{v}_1$};

\node at (11,0){s};
\node at (0, 16.5){w};

\node at (-0.3, 0.3)[scale=0.75]{0};
\node at (-0.3, 2)[scale=0.75]{2};
\node at (-0.3, 4)[scale=0.75]{4};
\node at (-0.3, 6)[scale=0.75]{6};
\node at (-0.3, 8)[scale=0.75]{8};
\node at (-0.4, 10)[scale=0.75]{10};
\node at (-0.4, 12)[scale=0.75]{12};
\node at (-0.4, 14)[scale=0.75]{14};
\node at (0.3, -2)[scale=0.75]{-2};
\node at (0.3, -4)[scale=0.75]{-4};
\node at (0.3, -6)[scale=0.75]{-6};
\node at (0.3, -8)[scale=0.75]{-8};
\node at (0.4, -10)[scale=0.75]{-10};

\node at (2, -0.3)[scale=0.75]{2};
\node at (4, -0.3)[scale=0.75]{4};
\node at (6, -0.3)[scale=0.75]{6};
\node at (8, -0.3)[scale=0.75]{8};
\node at (10, -0.3)[scale=0.75]{10};
\node at (-2, 0.3)[scale=0.75]{-2};
\node at (-4, 0.3)[scale=0.75]{-4};
\node at (-6, 0.3)[scale=0.75]{-6};

\foreach \z in {0, 1, 2, 3, 4, 5}
		\draw (\z+\z, \z-2) node{$\square$};

\foreach \z in {0, 1, 2, 3, 4, 5}
		\draw (\z+\z, \z-4) node{$\square$};

\foreach \z in {0, 1, 2, 3, 4, 5}
		\draw (\z+\z, \z-6) node{$\square$};

\foreach \z in {0, 1, 2, 3, 4, 5}
		\draw (\z+\z, \z-8) node{$\square$};

\foreach \z in {0, 1, 2, 3, 4, 5}
		\draw (\z+\z, \z-10) node{$\square$};

\foreach \z in { 2, 3, 4, 5}
		\draw (\z+\z, \z-12) node{$\square$};

\foreach \z in { 4, 5}
		\draw (\z+\z, \z-14) node{$\square$};

\foreach \z in {0, 1, 2, 3, 4, 5}
		\draw (\z+\z, \z) node{$\square$};

\foreach \z in {0, 1, 2, 3, 4, 5}
		\draw (\z+\z, \z+2) node{$\square$};

\foreach \z in {0, 1, 2, 3, 4, 5}
		\draw (\z+\z, \z+4) node{$\square$};

\foreach \z in {0, 1, 2, 3, 4, 5}
		\draw (\z+\z, \z+6) node{$\square$};

\foreach \z in {0, 1, 2, 3, 4, 5}
		\draw (\z+\z, \z+8) node{$\square$};

\foreach \z in {0, 1, 2, 3, 4, 5}
		\draw (\z+\z, \z+10) node{$\square$};

\foreach \z in {0, 1, 2, 3, 4}
		\draw (\z+\z, \z+12) node{$\square$};

\foreach \z in {0, 1, 2}
		\draw (\z+\z, \z+14) node{$\square$};

\node at (0.4, 1.7)[scale=0.75]{$\theta$};

\foreach \z in {1, 2, 3, 4, 5}
		\draw (\z+\z-1, \z+3) node{$\bullet$};
\foreach \z in {1, 2, 3, 4, 5, 6}
		\draw (\z+\z-2, \z+2) node{$\bullet$};

\node at (0,5){$\bullet$};
\node at (1,6){$\bullet$};

\foreach \z in {2,3, ...,10}
		\draw (\z-0.2, \z+5.25) node{$\bullet$};

\foreach \z in {1, 2, 3, 4, 5}
		\draw (\z+\z-1, \z+7) node{$\bullet$};
\foreach \z in {1, 2, 3, 4, 5, 6}
		\draw (\z+\z-2, \z+6) node{$\bullet$};

\node at (0,9){$\bullet$};
\node at (1,10){$\bullet$};
\foreach \z in {2,3, ...,7}
		\draw (\z-0.2, \z+9.25) node{$\bullet$};

\foreach \z in {1, 2, 3, 4, 5}
		\draw (\z+\z-1, \z+11) node{$\bullet$};
\foreach \z in {1, 2, 3, 4, 5, 6}
		\draw (\z+\z-2, \z+10) node{$\bullet$};

\node at (0,13){$\bullet$};
\node at (1,14){$\bullet$};
\node at (1.8,15.25){$\bullet$};
\node at (2.8,16.25){$\bullet$};

\node at (0,15){$\bullet$};
\node at (1,16){$\bullet$};
\node at (2,16){$\bullet$};

\node at (0.3, 2.7)[scale=0.75]{$\frac{\theta}{\tau}$};
\node at (0.3, 3.7)[scale=0.75]{$\frac{\theta}{\tau^2}$};
\node at (0.3, 4.7)[scale=0.75]{$\frac{\theta}{\tau^3}$};
\node at (0.3, 5.7)[scale=0.75]{$\frac{\theta}{\tau^4}$};
\node at (0.3, 6.7)[scale=0.75]{$\frac{\theta}{\tau^5}$};
\node at (0.3, 7.7)[scale=0.75]{$\frac{\theta}{\tau^6}$};
\node at (0.3, 8.7)[scale=0.75]{$\frac{\theta}{\tau^7}$};
\node at (0.3, 9.7)[scale=0.75]{$\frac{\theta}{\tau^8}$};
\node at (0.3, 10.7)[scale=0.75]{$\frac{\theta}{\tau^9}$};
\node at (0.3, 11.7)[scale=0.75]{$\frac{\theta}{\tau^{10}}$};
\node at (0.3, 12.7)[scale=0.75]{$\frac{\theta}{\tau^{11}}$};
\node at (0.3, 13.7)[scale=0.75]{$\frac{\theta}{\tau^{12}}$};
\node at (0.3, 14.7)[scale=0.75]{$\frac{\theta}{\tau^{13}}$};

\foreach \z in {0, 1, 2, 3, 4, 5}
		\draw (\z+\z-1, \z-1) node{$\bullet$};
\foreach \z in {0, 1, 2, 3, 4, 5, 6}
		\draw (\z+\z-2, \z-2) node{$\bullet$};
\node at (-3,-3){$\bullet$};
\node at (-4,-4){$\bullet$};
\node at (-5,-5){$\bullet$};
\node at (-6,-6){$\bullet$};

\node at (-1.2, -0.7)[scale=0.75]{$\rho$};
\node at (-2.2, -1.7)[scale=0.75]{$\rho^2$};
\node at (-3.2, -2.7)[scale=0.75]{$\rho^3$};
\node at (-4.2, -3.7)[scale=0.75]{$\rho^4$};
\node at (-5.2, -4.7)[scale=0.75]{$\rho^5$};
\node at (-6.2, -5.7)[scale=0.75]{$\rho^6$};

\foreach \z in {0, 1, 2, 3, 4, 5}
		\draw (\z+\z-1, \z-5) node{$\bullet$};
\foreach \z in {0, 1, 2, 3, 4, 5, 6}
		\draw (\z+\z-2, \z-6) node{$\bullet$};
\node at (-3,-7){$\bullet$};
\node at (-4,-8){$\bullet$};
\node at (-5,-9){$\bullet$};
\node at (-6,-10){$\bullet$};

\foreach \z in {0, 1, 2, 3, 4, 5}
		\draw (\z+\z-1, \z-9) node{$\bullet$};
\foreach \z in {0, 1, 2, 3, 4, 5,6}
		\draw (\z+\z-2, \z-10) node{$\bullet$};

\foreach \z in { 3, 4, 5}
		\draw (\z+\z-1, \z-13) node{$\bullet$};
\foreach \z in { 4, 5,6}
		\draw (\z+\z-2, \z-14) node{$\bullet$};

\node at (-0.4, -1.6)[scale=0.75]{$2\tau^2$};
\node at (-0.4, -3.6)[scale=0.75]{$\tau^4$};
\node at (-0.4, -5.6)[scale=0.75]{$2\tau^6$};
\node at (-0.4, -7.6)[scale=0.75]{$\tau^8$};
\node at (-0.4, -9.6)[scale=0.75]{$2\tau^{10}$};

\foreach \z in {0, 1, 2, 3, 4}
		\draw [red](\z+\z+0.15, \z+0.1) -- (\z+\z+1.85,\z+0.9);
\foreach \z in {0, 1, 2, 3, 4}
		\draw [red](\z+\z+0.15, \z-1.9) -- (\z+\z+1.85,\z-1.1);
\foreach \z in {0, 1, 2, 3, 4}
		\draw [red](\z+\z+0.15, \z-3.9) -- (\z+\z+1.85,\z-3.1);
\foreach \z in {0, 1, 2, 3, 4}
		\draw [red](\z+\z+0.15, \z-5.9) -- (\z+\z+1.85,\z-5.1);
\foreach \z in {0, 1, 2, 3, 4}
		\draw [red](\z+\z+0.15, \z-7.9) -- (\z+\z+1.85,\z-7.1);
\foreach \z in {0, 1, 2, 3, 4}
		\draw [red](\z+\z+0.15, \z-9.9) -- (\z+\z+1.85,\z-9.1);
\foreach \z in {2, 3, 4}
		\draw [red](\z+\z+0.15, \z-11.9) -- (\z+\z+1.85,\z-11.1);
\foreach \z in {4}
		\draw [red](\z+\z+0.15, \z-13.9) -- (\z+\z+1.85,\z-13.1);

\foreach \z in {0, 1, 2, 3, 4}
		\draw [red](\z+\z+0.15, \z+2.1) -- (\z+\z+1.85,\z+2.9);
\foreach \z in {0, 1, 2, 3, 4}
		\draw [red](\z+\z+0.15, \z+4.1) -- (\z+\z+1.85,\z+4.9);
\foreach \z in {0, 1, 2, 3, 4}
		\draw [red](\z+\z+0.15, \z+6.1) -- (\z+\z+1.85,\z+6.9);
\foreach \z in {0, 1, 2, 3, 4}
		\draw [red](\z+\z+0.15, \z+8.1) -- (\z+\z+1.85,\z+8.9);
\foreach \z in {0, 1, 2, 3, 4}
		\draw [red](\z+\z+0.15, \z+10.1) -- (\z+\z+1.85,\z+10.9);
\foreach \z in {0, 1, 2, 3}
		\draw [red](\z+\z+0.15, \z+12.1) -- (\z+\z+1.85,\z+12.9);
\foreach \z in {0, 1}
		\draw [red](\z+\z+0.15, \z+14.1) -- (\z+\z+1.85,\z+14.9);

\foreach \z in {0, 1, 2, 3, 4}
		\draw [red](\z+\z-1, \z-1) -- (\z+\z+1,\z);
        \draw [red](9, 4) -- (10, 4.5);
\foreach \z in {0, 1, 2, 3, 4, 5}
		\draw [red](\z+\z-2, \z-2) -- (\z+\z,\z-1);
\foreach \z in {0, 1, 2, 3, 4}
		\draw [red](\z+\z-1, \z-5) -- (\z+\z+1,\z-4);
        \draw [red](9, 0) -- (10, 0.5);
\foreach \z in {0, 1, 2, 3, 4, 5}
		\draw [red](\z+\z-2, \z-6) -- (\z+\z,\z-5);
\foreach \z in {0, 1, 2, 3, 4}
		\draw [red](\z+\z-1, \z-9) -- (\z+\z+1,\z-8);
        \draw [red](9, -4) -- (10, -3.5);
\foreach \z in {0, 1, 2, 3, 4, 5}
		\draw [red](\z+\z-2, \z-10) -- (\z+\z,\z-9);
\foreach \z in {3, 4}
		\draw [red](\z+\z-1, \z-13) -- (\z+\z+1,\z-12);
        \draw [red](9, -8) -- (10, -7.5);
\foreach \z in {4, 5}
		\draw [red](\z+\z-2, \z-14) -- (\z+\z,\z-13);

\foreach \z in {1, 2, 3, 4}
		\draw [red](\z+\z-1, \z+3) -- (\z+\z+1,\z+4);
        \draw [red](9, 8) -- (10, 8.5);
\foreach \z in {1, 2, 3, 4, 5}
		\draw [red](\z+\z-2, \z+2) -- (\z+\z,\z+3);
\foreach \z in {1, 2, 3, 4}
		\draw [red](\z+\z-1, \z+7) -- (\z+\z+1,\z+8);
        \draw [red](9, 12) -- (10, 12.5);
\foreach \z in {1, 2, 3, 4, 5}
		\draw [red](\z+\z-2, \z+6) -- (\z+\z,\z+7);
\foreach \z in {1, 2, 3, 4}
		\draw [red](\z+\z-1, \z+11) -- (\z+\z+1,\z+12);
        \draw [red](9, 16) -- (10, 16.5);
\foreach \z in {1, 2, 3, 4, 5}
		\draw [red](\z+\z-2, \z+10) -- (\z+\z,\z+11);
        \draw [red](0,15)--(2,16);

\foreach \z in {-1, -2, -3, -4, -5, -6}
		\draw (\z, \z) -- (\z+1,\z+1);
\foreach \z in {-1, -2, -3, -4, -5, -6}
		\draw (\z, \z-4) -- (\z+1,\z-3);
\foreach \z in {-1, -2}
		\draw (\z, \z-8) -- (\z+1,\z-7);

\foreach \z in {0, 1, 2, 3, 4}
        \draw (\z+\z+1, \z) -- (\z+\z+2, \z+1);
\foreach \z in {0, 1, 2, 3, 4}
        \draw (\z+\z, \z-1) -- (\z+\z+1, \z);
\foreach \z in {0, 1, 2, 3, 4}
        \draw (\z+\z+1, \z-4) -- (\z+\z+2, \z-3);
\foreach \z in {0, 1, 2, 3, 4}
        \draw (\z+\z, \z-5) -- (\z+\z+1, \z-4);
\foreach \z in {0, 1, 2, 3, 4}
        \draw (\z+\z+1, \z-8) -- (\z+\z+2, \z-7);
\foreach \z in {0, 1, 2, 3, 4}
        \draw (\z+\z, \z-9) -- (\z+\z+1, \z-8);
\foreach \z in { 2, 3, 4}
        \draw (\z+\z+1, \z-12) -- (\z+\z+2, \z-11);
\foreach \z in {3, 4}
        \draw (\z+\z, \z-13) -- (\z+\z+1, \z-12);

\foreach \z in {0, 1, 2, 3, 4}
        \draw (\z+\z+1, \z+4) -- (\z+\z+2, \z+5);
\foreach \z in {0, 1, 2, 3, 4}
        \draw (\z+\z, \z+3) -- (\z+\z+1, \z+4);
\foreach \z in {0, 1, 2, 3, 4}
        \draw (\z+\z+1, \z+8) -- (\z+\z+2, \z+9);
\foreach \z in {0, 1, 2, 3, 4}
        \draw (\z+\z, \z+7) -- (\z+\z+1, \z+8);
\foreach \z in {0, 1, 2, 3}
        \draw (\z+\z+1, \z+12) -- (\z+\z+2, \z+13);
\foreach \z in {0, 1, 2, 3, 4}
        \draw (\z+\z, \z+11) -- (\z+\z+1, \z+12);

\foreach \z in {5, 9, 13, 15}
        \draw (0, \z) -- (1, \z+1);
\foreach \z in {5, 9, 13}
        \draw (1, \z+1) -- (1.8, \z+2.25);

\foreach \z in {2, 3, 4, 5, 6, 7, 8, 9 }
        \draw (\z-0.2, \z+5.25)--(\z+0.8, \z+6.25);
\foreach \z in {2, 3, 4, 5, 6}
        \draw (\z-0.2, \z+9.25)--(\z+0.8, \z+10.25);
\foreach \z in {2}
        \draw (\z-0.2, \z+13.25)--(\z+0.8, \z+14.25);

\end{tikzpicture}
        \caption{The $RO(C_2)$-graded homotopy groups $\pi_{s, w}^{C_2}{k\mathbb R}$}
        \label{fig:kur}
    \end{figure}
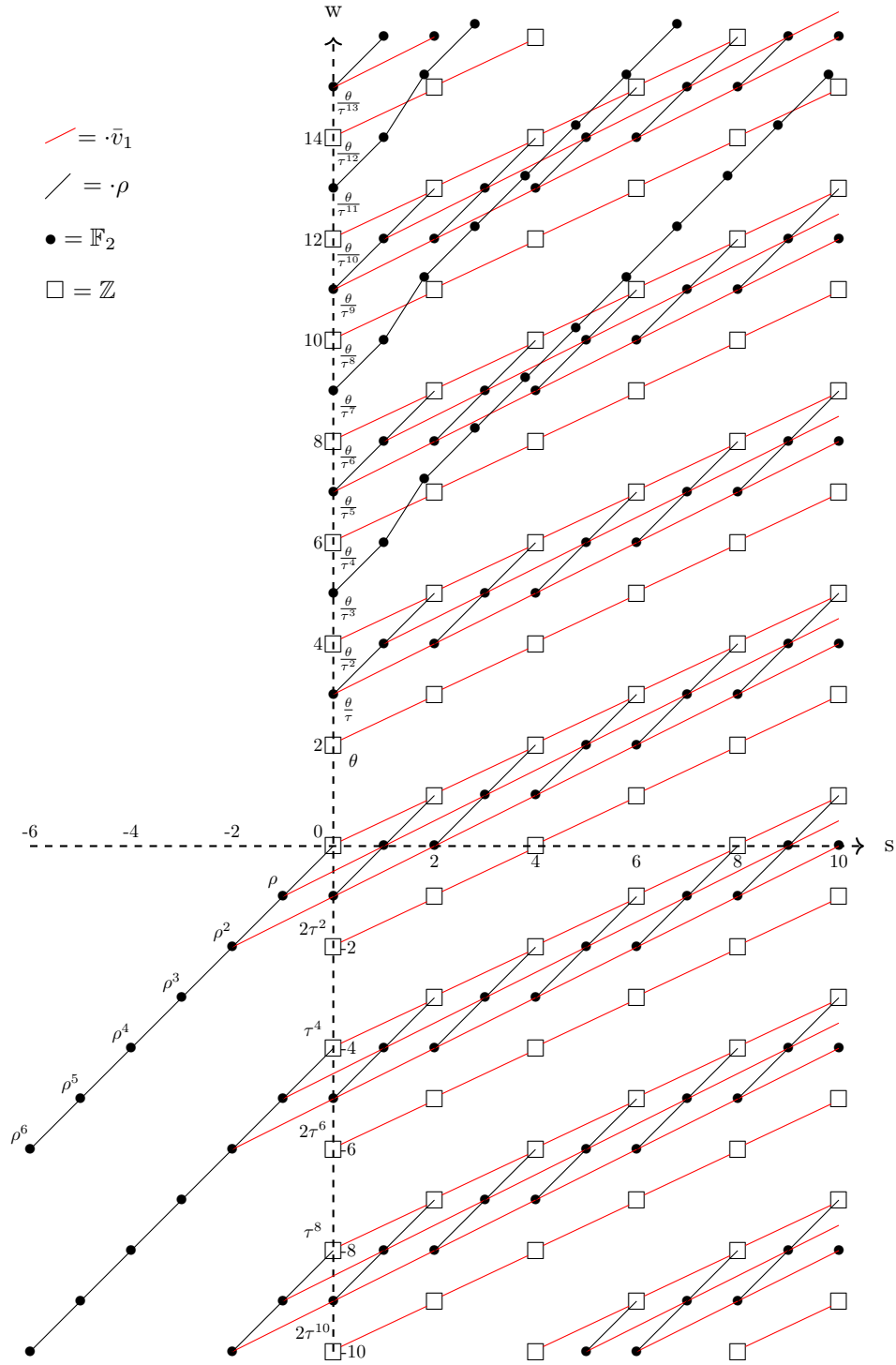

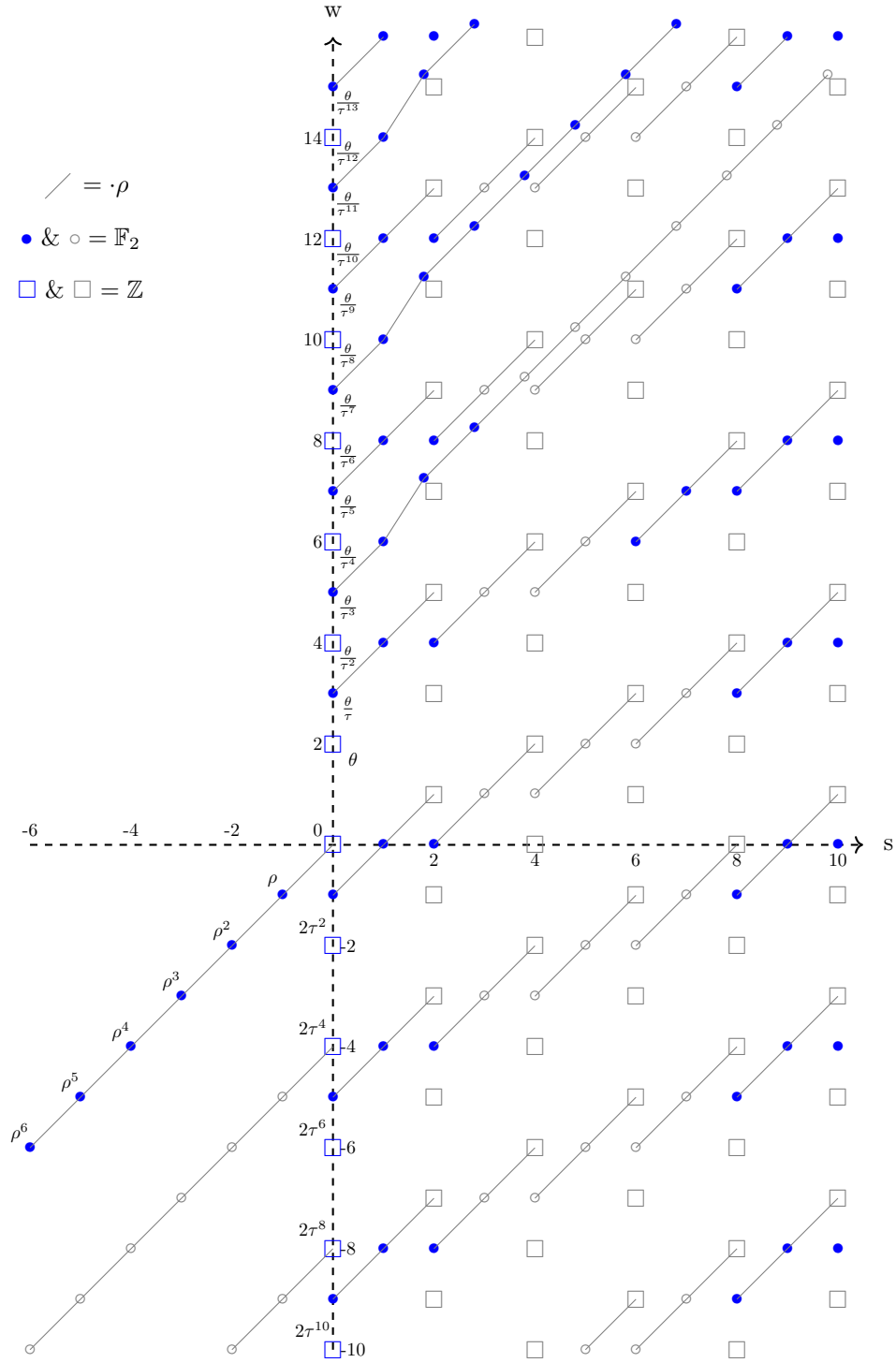
\begin{figure}[p]
        \centering
        \begin{tikzpicture}[scale=0.7135]
\draw[thick,dashed, ->] (-6,0) -- (10.5,0);
\draw[thick,dashed, ->] (0,-10) -- (0, 16);

\node at (-5, 12){$\color{blue}\bullet\ \color{black}\&\ \color{gray}\circ \color{black}= \bbF_2$};
\node at (-5, 11){$\color{blue}\square \ \color{black}\&\ \color{gray}\square$  $= \bbZ$};
\draw [gray](-5.7, 12.8)--(-5.2, 13.3);
\node at (-4.5, 13){$=\cdot \rho$};

\node at (11,0){s};
\node at (0, 16.5){w};

\node at (-0.3, 0.3)[scale=0.75]{0};
\node at (-0.3, 2)[scale=0.75]{2};
\node at (-0.3, 4)[scale=0.75]{4};
\node at (-0.3, 6)[scale=0.75]{6};
\node at (-0.3, 8)[scale=0.75]{8};
\node at (-0.4, 10)[scale=0.75]{10};
\node at (-0.4, 12)[scale=0.75]{12};
\node at (-0.4, 14)[scale=0.75]{14};
\node at (0.3, -2)[scale=0.75]{-2};
\node at (0.3, -4)[scale=0.75]{-4};
\node at (0.3, -6)[scale=0.75]{-6};
\node at (0.3, -8)[scale=0.75]{-8};
\node at (0.4, -10)[scale=0.75]{-10};

\node at (2, -0.3)[scale=0.75]{2};
\node at (4, -0.3)[scale=0.75]{4};
\node at (6, -0.3)[scale=0.75]{6};
\node at (8, -0.3)[scale=0.75]{8};
\node at (10, -0.3)[scale=0.75]{10};
\node at (-2, 0.3)[scale=0.75]{-2};
\node at (-4, 0.3)[scale=0.75]{-4};
\node at (-6, 0.3)[scale=0.75]{-6};

\foreach \z in {0, -2, -4, -6, -8, -10, 2, 4, 6, 8, 10, 12, 14}
        \draw [blue](0, \z) node{$\square$};

\foreach \z in {1, 2, 3, 4, 5}
		\draw [gray](\z+\z, \z-2) node{$\square$};

\foreach \z in {1, 2, 3, 4, 5}
		\draw [gray](\z+\z, \z-4) node{$\square$};

\foreach \z in {1, 2, 3, 4, 5}
		\draw [gray](\z+\z, \z-6) node{$\square$};

\foreach \z in {1, 2, 3, 4, 5}
		\draw [gray](\z+\z, \z-8) node{$\square$};

\foreach \z in {1, 2, 3, 4, 5}
		\draw [gray](\z+\z, \z-10) node{$\square$};

\foreach \z in { 2, 3, 4, 5}
		\draw [gray](\z+\z, \z-12) node{$\square$};

\foreach \z in { 4, 5}
		\draw [gray](\z+\z, \z-14) node{$\square$};

\foreach \z in {1, 2, 3, 4, 5}
		\draw [gray](\z+\z, \z) node{$\square$};

\foreach \z in {1, 2, 3, 4, 5}
		\draw [gray](\z+\z, \z+2) node{$\square$};

\foreach \z in {1, 2, 3, 4, 5}
		\draw [gray](\z+\z, \z+4) node{$\square$};

\foreach \z in {1, 2, 3, 4, 5}
		\draw [gray](\z+\z, \z+6) node{$\square$};

\foreach \z in {1, 2, 3, 4, 5}
		\draw [gray](\z+\z, \z+8) node{$\square$};

\foreach \z in { 1, 2, 3, 4, 5}
		\draw [gray](\z+\z, \z+10) node{$\square$};

\foreach \z in {1, 2, 3, 4}
		\draw [gray](\z+\z, \z+12) node{$\square$};

\foreach \z in { 1, 2}
		\draw [gray](\z+\z, \z+14) node{$\square$};

\node at (0.4, 1.7)[scale=0.75]{$\theta$};

\foreach \z in {3, 5, 7, 9, 11, 13, 15, -1, -5, -9}
        \draw[blue] (0, \z) node{$\bullet$};
\foreach \z in {3, 5, 7, 9, 11, 13, 15, -1, -5, -9}
        \draw[blue] (1, \z+1) node{$\bullet$};
\foreach \z in {3,  7,  11,  15, -1, -5, -9}
        \draw[blue] (2, \z+1) node{$\bullet$};

\foreach \z in {3,  7,  11,  15, -1, -5, -9}
        \draw[blue] (8, \z) node{$\bullet$};
\foreach \z in {3,  7,  11,  15, -1, -5, -9}
        \draw[blue] (9, \z+1) node{$\bullet$};
\foreach \z in {3,  7,  11,  15, -1, -5, -9}
        \draw[blue] (10, \z+1) node{$\bullet$};

\foreach \z in {2, 3}
		\draw (\z+\z-1, \z+3) node{$\color{gray}\circ$};
\foreach \z in {3}
		\draw (\z+\z-2, \z+2) node{$\color{gray}\circ$};

\foreach \z in {6, 7}
        \draw[blue] (\z, \z) node{$\bullet$};

\foreach \z in {2,3}
		\draw [blue](\z-0.2, \z+5.25) node{$\bullet$};
\foreach \z in {4, 5, ...,10}
		\draw (\z-0.2, \z+5.25) node{$\color{gray}\circ$};

\foreach \z in {2, 3, 4}
		\draw (\z+\z-1, \z+7) node{$\color{gray}\circ$};
\foreach \z in { 3, 4 }
		\draw (\z+\z-2, \z+6) node{$\color{gray}\circ$};

\foreach \z in {2,3, ...,7}
		\draw [blue](\z-0.2, \z+9.25) node{$\bullet$};

\foreach \z in {2, 3, 4}
		\draw (\z+\z-1, \z+11) node{$\color{gray}\circ$};
\foreach \z in {3, 4}
		\draw (\z+\z-2, \z+10) node{$\color{gray}\circ$};

\node at (1.8,15.25){$\color{blue}\bullet$};
\node at (2.8,16.25){$\color{blue}\bullet$};

\node at (0.3, 2.7)[scale=0.75]{$\frac{\theta}{\tau}$};
\node at (0.3, 3.7)[scale=0.75]{$\frac{\theta}{\tau^2}$};
\node at (0.3, 4.7)[scale=0.75]{$\frac{\theta}{\tau^3}$};
\node at (0.3, 5.7)[scale=0.75]{$\frac{\theta}{\tau^4}$};
\node at (0.3, 6.7)[scale=0.75]{$\frac{\theta}{\tau^5}$};
\node at (0.3, 7.7)[scale=0.75]{$\frac{\theta}{\tau^6}$};
\node at (0.3, 8.7)[scale=0.75]{$\frac{\theta}{\tau^7}$};
\node at (0.3, 9.7)[scale=0.75]{$\frac{\theta}{\tau^8}$};
\node at (0.3, 10.7)[scale=0.75]{$\frac{\theta}{\tau^9}$};
\node at (0.3, 11.7)[scale=0.75]{$\frac{\theta}{\tau^{10}}$};
\node at (0.3, 12.7)[scale=0.75]{$\frac{\theta}{\tau^{11}}$};
\node at (0.3, 13.7)[scale=0.75]{$\frac{\theta}{\tau^{12}}$};
\node at (0.3, 14.7)[scale=0.75]{$\frac{\theta}{\tau^{13}}$};

\foreach \z in {2, 3, 4}
		\draw (\z+\z-1, \z-1) node{$\color{gray}\circ$};
\foreach \z in {3, 4}
		\draw (\z+\z-2, \z-2) node{$\color{gray}\circ$};
\node at (-1,-1){$\color{blue}\bullet$};
\node at (-2,-2){$\color{blue}\bullet$};
\node at (-3,-3){$\color{blue}\bullet$};
\node at (-4,-4){$\color{blue}\bullet$};
\node at (-5,-5){$\color{blue}\bullet$};
\node at (-6,-6){$\color{blue}\bullet$};

\node at (-1.2, -0.7)[scale=0.75]{$\rho$};
\node at (-2.2, -1.7)[scale=0.75]{$\rho^2$};
\node at (-3.2, -2.7)[scale=0.75]{$\rho^3$};
\node at (-4.2, -3.7)[scale=0.75]{$\rho^4$};
\node at (-5.2, -4.7)[scale=0.75]{$\rho^5$};
\node at (-6.2, -5.7)[scale=0.75]{$\rho^6$};

\foreach \z in {0, 2, 3, 4}
		\draw (\z+\z-1, \z-5) node{$\color{gray}\circ$};
\foreach \z in {0, 3, 4}
		\draw (\z+\z-2, \z-6) node{$\color{gray}\circ$};
\node at (-3,-7){$\color{gray}\circ$};
\node at (-4,-8){$\color{gray}\circ$};
\node at (-5,-9){$\color{gray}\circ$};
\node at (-6,-10){$\color{gray}\circ$};

\foreach \z in {0, 2, 3, 4}
		\draw (\z+\z-1, \z-9) node{$\color{gray}\circ$};
\foreach \z in {0, 3, 4}
		\draw (\z+\z-2, \z-10) node{$\color{gray}\circ$};

\foreach \z in { 3, 4}
		\draw (\z+\z-1, \z-13) node{$\color{gray}\circ$};
\foreach \z in { 4}
		\draw (\z+\z-2, \z-14) node{$\color{gray}\circ$};

\node at (-0.4, -1.6)[scale=0.75]{$2\tau^2$};
\node at (-0.4, -3.6)[scale=0.75]{$2\tau^4$};
\node at (-0.4, -5.6)[scale=0.75]{$2\tau^6$};
\node at (-0.4, -7.6)[scale=0.75]{$2\tau^8$};
\node at (-0.4, -9.6)[scale=0.75]{$2\tau^{10}$};

\foreach \z in {-1, -2, -3, -4, -5, -6}
		\draw [gray](\z, \z) -- (\z+1,\z+1);
\foreach \z in {-1, -2, -3, -4, -5, -6}
		\draw [gray](\z, \z-4) -- (\z+1,\z-3);
\foreach \z in {-1, -2}
		\draw [gray](\z, \z-8) -- (\z+1,\z-7);

\foreach \z in {0, 1, 2, 3, 4}
        \draw [gray](\z+\z+1, \z) -- (\z+\z+2, \z+1);
\foreach \z in {0, 1, 2, 3, 4}
        \draw [gray](\z+\z, \z-1) -- (\z+\z+1, \z);
\foreach \z in {0, 1, 2, 3, 4}
        \draw [gray](\z+\z+1, \z-4) -- (\z+\z+2, \z-3);
\foreach \z in {0, 1, 2, 3, 4}
        \draw [gray](\z+\z, \z-5) -- (\z+\z+1, \z-4);
\foreach \z in {0, 1, 2, 3, 4}
        \draw [gray](\z+\z+1, \z-8) -- (\z+\z+2, \z-7);
\foreach \z in {0, 1, 2, 3, 4}
        \draw [gray](\z+\z, \z-9) -- (\z+\z+1, \z-8);
\foreach \z in { 2, 3, 4}
        \draw [gray](\z+\z+1, \z-12) -- (\z+\z+2, \z-11);
\foreach \z in {3, 4}
        \draw [gray](\z+\z, \z-13) -- (\z+\z+1, \z-12);

\foreach \z in {0, 1, 2, 3, 4}
        \draw [gray](\z+\z+1, \z+4) -- (\z+\z+2, \z+5);
\foreach \z in {0, 1, 2, 3, 4}
        \draw [gray](\z+\z, \z+3) -- (\z+\z+1, \z+4);
\foreach \z in {0, 1, 2, 3, 4}
        \draw [gray](\z+\z+1, \z+8) -- (\z+\z+2, \z+9);
\foreach \z in {0, 1, 2, 3, 4}
        \draw [gray](\z+\z, \z+7) -- (\z+\z+1, \z+8);
\foreach \z in {0, 1, 2, 3}
        \draw [gray](\z+\z+1, \z+12) -- (\z+\z+2, \z+13);
\foreach \z in {0, 1, 2, 3, 4}
        \draw [gray](\z+\z, \z+11) -- (\z+\z+1, \z+12);

\foreach \z in {5, 9, 13, 15}
        \draw [gray](0, \z) -- (1, \z+1);
\foreach \z in {5, 9, 13}
        \draw [gray](1, \z+1) -- (1.8, \z+2.25);

\foreach \z in {2, 3, 4, 5, 6, 7, 8, 9 }
        \draw [gray](\z-0.2, \z+5.25)--(\z+0.8, \z+6.25);
\foreach \z in {2, 3, 4, 5, 6}
        \draw [gray](\z-0.2, \z+9.25)--(\z+0.8, \z+10.25);
\foreach \z in {2}
        \draw [gray](\z-0.2, \z+13.25)--(\z+0.8, \z+14.25);

\end{tikzpicture}
        \caption{The $RO(C_2)$-graded Hurewicz image of $k\bbR$}
        \label{fig:hurewiczkur}
    \end{figure}

\afterpage{\clearpage}

The explicit generators will be specified in \cref{thm:krhurewicz}. A diagram of the Hurewicz image of $k\bbR$ is shown in \cref{fig:hurewiczkur}; see \cref{sec:hurewiczKR} for a detailed discussion.

In \cref{thm:mainhurewicz}, $\psi(n)$ denotes the $n$-th Radon--Hurwitz number \cite{eckmann1942gruppen}; in particular, Adams proved that the maximal number of continuous linearly independent vector fields on $S^{n-1}$ is $\psi(n)-1$ in \cite{adams1962vector}, and this governs the Hurewicz image of $H\underline{\bbF_2}$ and $H\underline{\bbZ}$ as shown in \cite{gllmwxyz2026hurewicz}. As the infinite $\rho$-divisible pattern in the negative cone of $k\bbR$ maps isomorphically to that of $H\underline{\bbF_2}$ and $H\underline{\bbZ}$, this part can be read off directly from the result in \cite{gllmwxyz2026hurewicz}. More precisely, by \cite[Theorem~7.2(1)]{gllmwxyz2026hurewicz}, the composite
\[
     S^{0,0}\longrightarrow k\bbR \longrightarrow H\underline{\bbZ}
\]
accounts for the stem-zero classes, the negative diagonal, and the classes involving Radon--Hurwitz numbers in every negative odd coweight. If $s-w\equiv3 \pmod4$, these are precisely the classes on the last nonzero line of \cref{thm:mainhurewicz}. If $s-w\equiv1\pmod4$ and $s-w<0$, then $w-s-1\equiv2\pmod4$, so $\psi(w-s-1)=2$. Hence, the $H\underline{\bbZ}$-image is nonzero only for $s=0,1$, corresponding to the $l=0$ cases of the coweight congruent to $1 \text{ mod } 4$ family. 

The remaining cases constitute the new height-one periodic contribution. These classes exhibit both $\tau^4$-periodicity, corresponding to the James periodicity of truncated projective spectra, and $\bar{v}_1$-periodicity, which is an equivariant lift of the classical $v_1$-periodicity. The analysis of these classes is mainly based on the Adams spectral sequences of $k\bbR$ and $ko_{C_2}$, as computed in \cite{guillouhillisaksenravenel2020cohomology}. Since there is a factorization
    \begin{equation}\label{eq: kotokR}
            S^{0,0} \overset{\iota_{ko_{C_2}}}{\longrightarrow} ko_{C_2}\overset{c}{\longrightarrow} k\bbR,
    \end{equation}
\cref{thm:mainhurewicz} also gives partial information about the $RO(C_2)$-graded Hurewicz image of $ko_{C_2}$. In particular, our proof of \cref{thm:mainhurewicz} uses the information about the complexification map $c: ko_{C_2}\to k\bbR$ as computed in \cite{guillouhillisaksenravenel2020cohomology}. The full $RO(C_2)$-graded Hurewicz image of $ko_{C_2}$ will be treated in a sequel.

Notably, these $\tau^4$- and $\bar{v}_1$-periodic classes in $\pi_\star^{C_2}S^{0, 0}$ with nontrivial $k\bbR$-Hurewicz images are lifts of the classical $ko$-Hurewicz images, whose computation led to the construction and detection of the $\mu$-families in the classical stable stems \cite{adams1966on}. In particular, we have the following families of $C_2$-equivariant lifts of the classical $\mu$-family elements.

\begin{theorem}[\cref{cor: four periodic families}]\label{thm:mainmufamily}
~\begin{enumerate}
    \item For $k, l\geq 0$, there are two families of elements in $\pi_\star^{C_2}S^{0,0}$ at degree
    \begin{itemize}
        \item $(8l+1, 4(l-k))$,
        \item $(8l+1, 4(l+k+1))$ with $l>0$,
    \end{itemize}
    whose underlying is the classical $\mu_{8l+1}$-family in $\pi_*^s$. 
    \item For $k, l\geq 0$, there are two families of elements in $\pi_\star^{C_2}S^{0,0}$ at degree
    \begin{itemize}
        \item $(8l+2, 4(l-k))$,
        \item $(8l+2, 4(l+k+1))$,
    \end{itemize}
    whose underlying is the classical $\mu_{8l+2}$-family in $\pi_*^s$.
\end{enumerate}    
All four families have nontrivial $k\bbR$-Hurewicz images.
\tqed
\end{theorem}

In \cref{sec4}, we explicitly construct the first family in \cref{thm:mainmufamily}(1). We do so by constructing an equivariant lift of the classical Adams periodicity operator $P$, coming from a $v_{1,0}$-self map on $S^{0,0}/2$. See \cref{sec4.1} for details.

\begingroup
\renewcommand{\addcontentsline}[3]{}
\subsection*{Future Directions}
We view the computations in this paper and in \cite{gllmwxyz2026hurewicz} as the first steps in a program to determine the $RO(C_2)$-graded Hurewicz images of the truncated Real Brown--Peterson spectra $BP_{\bbR}\langle n\rangle$. These spectra fit into the tower
\[
S^{0,0}\to BP_{\bbR}\to\cdots\to BP_{\bbR}\langle2\rangle\to BP_{\bbR}\langle1\rangle\simeq k\bbR\to BP_{\bbR}\langle0\rangle\simeq H\underline{\bbZ}\to H\underline{\bbF_2}.
\]
The goal is to understand which equivariant stable homotopy classes these theories detect and to use these computations to construct periodic families.

These computations already give interesting information about the $C_2$-equivariant sphere. The Hurewicz images of $H\underline{\bbZ}$ and $H\underline{\bbF_2}$ are closely tied to vector fields on spheres. Our computation of the $k\bbR$-Hurewicz image gives an $RO(C_2)$-graded refinement of the classical height-one $ko$-Hurewicz image, detecting four infinite families that lift the classical $\mu$-families. We construct explicit representatives for one of these families using an equivariant Adams periodicity operator.

Li--Shi--Wang--Xu \cite{lishiwangxu2019hurewicz} studied the classical Hurewicz maps from $S^0$ to $BP_{\bbR}^{C_2}$ and $E_{\bbR}(n)^{C_2}$. Their detection theorems for $E_{\bbR}(n)^{C_2}$ also give detections for $BP_{\bbR}\langle n\rangle^{C_2}$, since $E_{\bbR}(n)$ is obtained from $BP_{\bbR}\langle n\rangle$ by inverting $\bar{v}_n$. These results motivate the following three questions.

\begin{enumerate}[label=\arabic*., leftmargin=*, itemsep=\smallskipamount]
\item \textbf{What geometric conditions characterize the Hurewicz image of $BP_{\bbR}\langle n\rangle$?}
The $H\underline{\bbF_2}$- and $H\underline{\bbZ}$-images are governed by vector fields on spheres. Are there analogous geometric existence and obstruction criteria that characterize the Hurewicz images at higher stages?

\item \textbf{What new equivariant stable homotopy classes are detected by $BP_{\bbR}\langle n\rangle$?}
For $n\geq 2$, we would like to determine the full $RO(C_2)$-graded Hurewicz images and their relation to the classical detections of Li--Shi--Wang--Xu \cite{lishiwangxu2019hurewicz}.

\item \textbf{Can these computations be used to construct new periodic families in the equivariant stable stems?}
Do there exist higher-height analogues of the periodicity operator constructed in this paper, and do they produce infinite families in $\pi_\star^{C_2} S^{0,0}$ that are detected in $BP_{\bbR}\langle n\rangle$?
\end{enumerate}
\par
\endgroup

\subsection*{Organization}
In \cref{sec:prelim}, we recall the necessary $C_2$-equivariant background. In \cref{sec:hurewiczKR}, we compute the Hurewicz image of $k\bbR$ and prove \cref{thm:mainhurewicz}. In \cref{sec4}, we construct a $C_2$-equivariant $v_{1,0}$-self map and an associated periodicity operator, which produces an infinite family of elements in the $C_2$-equivariant stable stems that lifts the classical $\mu$-family.

\subsection*{Acknowledgment} 
We would like to thank Mike Hill, Hana Jia Kong, and Guchuan Li for helpful conversations. The second author is partially supported by NSF Grant No.~DMS-2404828 and NSF CAREER Grant No.~DMS-2541934. The third author is partially supported by the AMS Centennial Research Fellowship and NSF Grant DMS-2506247. The fifth author gratefully acknowledges that this research is supported in part by the Pacific Institute for the Mathematical Sciences.


\section{Preliminaries}\label{sec:prelim}

Throughout, all spectra are implicitly 2-completed. We review the basics of the $\bbR$-motivic and genuine $C_2$-equivariant Adams spectral sequences, and the connective $C_2$-equivariant Atiyah's Real $K$-theory. Recall that the bidegree $(s, w)$ represents the $RO(C_2)$-degree $(s-w) + w \sigma$. In the $\bbR$-motivic setting, a bidegree $(s,w)$ represents the pair consisting of the stem and the motivic weight.

Let $H\bbF_2^\bbR$ denote the $\bbR$-motivic Eilenberg--MacLane spectrum with $\bbF_2$-coefficients; let $\mathcal{A}^{\bbR}$ and $\mathcal{A}^{C_2}$ denote the $\bbR$-motivic and genuine $C_2$-equivariant Steenrod algebras, respectively. Then, the $\bbR$-motivic and genuine $C_2$-equivariant Adams spectral sequences of the spheres are
\begin{align*}
\leftindex^{\bbR}E_2^{s, f, w}&=\Ext_{\bbR}^{s,f,w}:=\Ext_{\calA^{\bbR}}^{s,f,w}({H\bbF_2}^\bbR_{*, *}, {H\bbF_2}^\bbR_{*, *})\implies (\pi_{s,w}^\bbR)^\wedge_2\\
\leftindex^{C_2}E_2^{s, f, w}&= \Ext_{C_2}^{s,f,w}:=\Ext_{\calA^{C_2}}^{s,f,w}({H\underline{\bbF_2}}_{*, *}, {H\underline{\bbF_2}}_{*, *}) \implies (\pi_{s,w}^{C_2})^\wedge_2
\end{align*}
where $f$ denotes the Adams filtration. References for the Adams spectral sequences mentioned above can be found in \cite{duggerisaksen2010motivic,duggerisaksen2017low, guillouhillisaksenravenel2020cohomology, ma2026borel}.

Let $k\bbR$ denote the connective cover of $K\bbR$, as discussed in \cite{guillouhillisaksenravenel2020cohomology}. We recall the full $RO(C_2)$-graded homotopy groups of $k\bbR$ as follows.

\begin{prop}[\cite{dugger2005atiyah, brunergreenlees2010connective, greenleesmeier2017gorenstein}] \label{thm: ku_R computation}
    The $RO(C_2)$-graded homotopy groups of $k \bbR$ are given by
    \[
        \pi_{\star}^{C_2} k \bbR = PC \oplus NC,
    \]
    where 
    \begin{align*}
        & PC = \frac{\bbZ [\rho, \tau^4, \bar{v}_1] }{ (2 \rho, \bar{v}_1 \rho^3)} \{ 1 \} \oplus \bbZ[\tau^4, \bar{v}_1] \{2\tau^2\}, \\
        & NC = \bigoplus_{n \geq 0} \bbZ [\bar{v}_1] \left \{ \frac{\theta}{\tau^{4n}} \right \} \oplus \bigoplus_{n,k \geq 0} \bbF_2 \left \{ \frac{\theta}{\rho^k \tau^{4n+3}} \right \} \\
         &\oplus \bigoplus_{\substack{n \geq 0 }} \bbZ [\bar{v}_1 ] \left \{ \frac{\theta}{\tau^{4n+2}}, \frac{\theta}{\rho^{k} \tau^{4n+1}} \right \}_{k=0,1,2}  {\mathlarger{/}} \left ( 2\cdot \frac{\theta}{\rho^\epsilon \tau^{4n+1}},\  2\cdot \frac{\theta}{\rho^{2} \tau^{4n+1}}  - \frac{\theta}{\tau^{4n+2}}\bar{v}_1 \right ), \ \epsilon=0, 1.
    \end{align*}
    where generators have bidegrees
    \[|\rho|=(-1, -1), \quad |\theta|=(0, 2), \quad |\tau|=(0, -1), \quad |\bar{v}_1|=(2, 1).\]
    The summands $PC$ and $NC$ are referred to as the positive cone and the negative cone, respectively. 
    \tqed
\end{prop}

The $RO(C_2)$-graded homotopy groups $\pi_{s, w}^{C_2}k\bbR$ are depicted in \cref{fig:kur}. Each square represents a generator of $\bbZ$, while each dot represents a generator of $\bbF_2$. The black lines of slope 1 indicate multiplication by $\rho$ from the upper right to the lower left. The red lines of slope $1/2$ indicate multiplication by $\bar{v}_1$ from the lower left to the upper right. The names of the generators correspond, under the composite map 
\[
    k \bbR \to H \underline{\bbZ} \to H \underline{\bbF_2},
\]
to the generators with the same names in $H \underline{\bbZ}$ and $H \underline{\bbF_2}$. Here, the first map is the canonical map to the 0-slice, and the second map is induced by the quotient map $\bbZ \to \bbF_2$. Note that $\tau$ is not an actual element in $\pi_\star^{C_2}k\bbR$. The generator $\bar{v}_1$ is an equivariant lift of the classical generator of $\pi_2 ku$. For a detailed discussion of these generators, we refer the reader to \cite[Section 3.1]{lishiwangxu2019hurewicz}.

\begin{rmk}
    We choose the names of our generators in \cref{thm: ku_R computation} so that they are compatible with the map $k\bbR \to H\underline{\bbF_2}$, which is convenient for comparing the Hurewicz images of these two spectra and relating them to the $\bbR$-motivic setting. Some generators may have different names in other references. For example, in \cite{hillhopkinsravenel2016nonexistence, lishiwangxu2019hurewicz}, the generators $\rho$ and $\tau$ are denoted by $a_{\sigma}$ and $u_{\sigma}$, respectively. In this paper, we will use $a_{\sigma}$ to denote the Euler class of $\sigma$ in $\pi_\star^{C_2} S^{0,0}$.
\end{rmk}

\section{The Hurewicz image of \texorpdfstring{$k \bbR$}{kR}} \label{sec:hurewiczKR}
We prove \cref{thm:mainhurewicz} in this section. The following restatement specifies the generators explicitly.

\begin{theorem}\label{thm:krhurewicz}
   The $RO(C_2)$-graded Hurewicz image of $k\bbR$ consists of
\[\im (\iota_{k\bbR} )_{s, w}=\begin{cases}
        \bbF_2\{\rho^{-s}\} \quad &\mathrm{if}\ s=w<0,\\
        \bbZ\{1\} \quad &\mathrm{if}\ s=w=0,\\
        \bbZ\{\frac{\theta}{\tau^{2k-2}}\} \quad &\mathrm{if}\ s=0\ \mathrm{and}\ w=2k\ \mathrm{for}\ k > 0,\\
        \bbZ\{2\tau^{2k}\} \quad &\mathrm{if}\ s=0\ \mathrm{and}\ w=-2k\ \mathrm{for}\ k>0,\\
        \bbF_2\{\frac{\theta}{\rho^\epsilon\tau^{4k+1}}\bar{v}_1^{4k+3}\} \quad &\mathrm{if}\ s=w=8k+6+\epsilon, \\ &\ \ \ \mathrm{for}\ k\geq 0, \epsilon =0, 1,\\
        \bbF_2\{\rho^{2-\epsilon}\tau^{4k}\bar{v}_1^{4l+1}\} & \mathrm{if}\ s=8l+\epsilon\ \mathrm{and}\ s-w=4(k+l)+1 \\ &\ \ \ \mathrm{for}\ k\geq 0, l\geq 0, \epsilon =0, 1, \\
        \bbF_2\{\frac{\theta}{\rho^\epsilon\tau^{-4k-3}}\bar{v}_1^{4l}\} & \mathrm{if}\ s=8l+\epsilon\ \mathrm{and}\ s-w=4(k+l)+1\\ &\ \ \  \mathrm{for}\ k< 0, l\geq 0, \epsilon =0, 1,\\
        \bbF_2\{\rho^{2}\tau^{4k}\bar{v}_1^{4l+2}\} & \mathrm{if}\ s=8l+2 \ \mathrm{and}\ s-w=4(k+l)+2\\ &\ \ \  \mathrm{for}\ k\geq 0, l\geq 0, \\
        \bbF_2\{\frac{\theta}{\tau^{-4k-3}}\bar{v}_1^{4l+1}\} & \mathrm{if}\ s=8l+2 \ \mathrm{and}\ s-w=4(k+l)+2\\ &\ \ \  \mathrm{for}\ k< 0, l\geq 0,\\
        \bbF_2\{\frac{\theta}{\rho^s\tau^{w-s-2}}\} \quad &\mathrm{if}\ 0\leq s< \psi(w-s-1) \ \mathrm{and}\ s-w=4k-1\\ &\ \ \  \mathrm{for\ } k\leq -1,\\
        0\quad &\mathrm{otherwise.}
    \end{cases}\]
\end{theorem}

The results and the corresponding proofs are summarized in \cref{table:hurewiczImage}. A diagram illustrating \cref{thm:krhurewicz} can be found in \cref{fig:hurewiczkur}. For readers' convenience, we highlight the following key features of the chart:

\begin{itemize}
    \item Each square represents a generator of $\bbZ$, while each dot represents a generator of $\bbF_2$.
    \item Lines of slope 1 indicate multiplication by $\rho$.
    \item Blue squares and dots are precisely the elements in the Hurewicz image.
    \item Lines indicating multiplication by $\bar{v}_1$ have been omitted compared to \cref{fig:kur}.
\end{itemize}

\begin{table}[ht]
    \centering
    \captionsetup{skip=10pt}
    \setlength{\arrayrulewidth}{1pt}
    \begin{tabular}{c|c|c|c}
        \hline
        \rule{0pt}{3ex} $(s,w)$ & Hurewicz image & in $\pi_\star^{C_2}$ & proof \\ \hline
            $(8l+1-\epsilon, 4(l-k)-\epsilon)$    &  $\rho^{1+\epsilon} \tau^{4k} \bar{v}_1^{4l+1}$  & $\{\rho^{\epsilon}\tau^{4k+1} P^{l}h_1\}$  & \cref{lem:combinetauh1}(1)  \\
            $(8l+2, 4(l-k))$    &   $\rho^2 \tau^{4k} \bar{v}_1^{4l+2}$ & $\{\tau^{4k+2} P^l h_1^2 \}$ & \cref{lem:combine(tauh1)2}(1) \\
            $(8k+6 + \epsilon, 8k+6 + \epsilon)$    &   $\frac{\theta}{\rho^\epsilon\tau^{4k+1}}\bar{v}_1^{4k+3}$ & $\frac{\eta^{4k+4}}{\rho^{4k+2+\epsilon}}$ & \cref{lem:cowt0} \\
            $(8l + \epsilon, 4(l+k)+3+\epsilon)$    &   $\frac{\theta}{\rho^\epsilon\tau^{4k+1}}\bar{v}_1^{4l}$  & $\{\frac{\theta}{\rho^{1+\epsilon}\tau^{4k}}x_{8l-1}\}$ & \cref{lem:combinetauh1}(2) \\
            $(8l+2, 4(l+k+1))$    &   $\frac{\theta}{\tau^{4k+1}}\bar{v}_1^{4l+1}$ & $\{\frac{\theta}{\rho\tau^{4k}}P^l h_1\}$ & \cref{lem:combine(tauh1)2}(2) \\
            $(s,w)$    &   $\frac{\theta}{\rho^s\tau^{w-s-2}}$ & & \cite[Theorem~1.2]{gllmwxyz2026hurewicz} \\ 
            $(-s, -s)$    &   $\rho^{s}$ & $a_{\sigma}^{s}$ & \cite[Theorem~1.2]{gllmwxyz2026hurewicz}  \\
            $(0, 2k)$    &   $\frac{\theta}{\tau^{2k-2}}$    &  $\omega_{-k}$  & \cite[Theorem~7.2]{gllmwxyz2026hurewicz} \\
            $(0, -2k)$    &   $2\tau^{2k}$    &  $\omega_k$ & \cite[Theorem~7.2]{gllmwxyz2026hurewicz} \\
            \hline
    \end{tabular}
    \caption{The $RO(C_2)$-graded Hurewicz image of $k\bbR$}
    \label{table:hurewiczImage}
\end{table}

Since the Hurewicz map of the $C_2$-equivariant Eilenberg--MacLane spectrum $H\underline{\bbZ}$ factors through $k\bbR$,
\[\iota_{H\underline{\bbZ}}: S^{0,0} \overset{\iota_{k\bbR}}{\longrightarrow} k\bbR \longrightarrow H\underline{\bbZ},\]
by \cite[Theorem 7.2(1)]{gllmwxyz2026hurewicz}, the known Hurewicz image of $H\underline{\bbZ}$ forces the result of \cref{thm:krhurewicz} at degrees $s\leq 0$, as well as $s-w=4k-1$ for $k\leq -1$.

\begin{lemma}\label{lem:kohurewicz}
    For $k\geq 0$, the classes $\rho\tau^{4k}\bar{v}_1^{4k+1}, \rho^2\tau^{4k}\bar{v}_1^{4k+2}$ are in the Hurewicz image of $k\bbR$.
\end{lemma}
\begin{proof}
    Note that the classes in the lemma have bidegrees $(8k+1, 0)$ and $(8k+2, 0)$, respectively. Since their weight is $0$, they can be identified as generators of 
    \[\pi_* \bigl((k\bbR)^{C_2} \bigr) =\pi_*ko.\]
    The claim then follows from the classical Hurewicz image of $ko$.
\end{proof}

Comparison with the computation of the genuine $C_2$-equivariant Adams spectral sequence of $S^{0,0}$ \cite{guillouisaksen2024c2} shows that the classes $\rho\bar{v}_1$ in degree $(1, 0)$ and $\rho^2\bar{v}_1^2$ in degree $(2, 0)$ are the Hurewicz images of $\{\tau h_1\}$ and $ \{(\tau h_1)^2\}$ in $\pi_\star^{C_2}S^{0,0}$, respectively. Since $\rho\tau h_1$ is a permanent cycle, $\rho^2 \bar{v}_1$ is also in the Hurewicz image. 

\cref{lem:kohurewicz} further implies that there are two infinite families in $\pi_\star^{C_2} S^{0,0}$ detected by $k\bbR$, which are lifts of Adams' $\mu$-family in the classical stable stems. Comparing with the genuine $C_2$-equivariant Adams spectral sequence, the only possible candidates are the families $\{\tau^{4k+1}P^{k}h_1\}$ and $\{\tau^{4k+2}P^{k}h_1^2\}$ for $k\geq 0$.

Moreover, by \cite{guillouisaksen2024c2}, \cite{cohenlinmahowald1988adams}, and \cite{ma2026borel}, there are $\rho$-Bockstein differentials
\begin{equation}\label{eq:rhobock}
    d_2(\tau^{4k+2}) =\rho^2 \tau^{4k+1} h_1
\end{equation}
and for $k>0$, 
\begin{equation} \label{eqn:rho_Bockstein}
    d_3(\tau^{4k+2}x_{8k-1})=\rho^3\tau^{4k+1} P^{k}h_1.
\end{equation}
where $x_{8k-1}$ is the order 2 element at stem $8k-1$ of the 2-primary image of $J$. Comparing with the computations in \cite{guillouhillisaksenravenel2020cohomology}, we have that $\rho\tau^{4k+1}P^kh_1$ is also a surviving permanent cycle for all $k\geq 0$, with nontrivial image $\rho^2\tau^{4k}\bar{v}_1^{4k+1}$ in $k\bbR$.

\cref{thm:krhurewicz} and \cref{fig:hurewiczkur} indicate that, besides the classes in $\im(\iota_{H\underline{\bbZ}})$ and exceptions at coweight 0, the Hurewicz image of $k\bbR$ is generated by the classes in \cref{lem:kohurewicz} under $\tau^4$-periodicity. The next few lemmas are devoted to proving this.

\begin{lemma}\label{lem:combinetauh1}
    Let $\epsilon=0$ or $1$. For $k, l\geq 0$,
\begin{enumerate}
\item the classes
\[
\rho^{1+\epsilon}\tau^{4k}\bar{v}_1^{4l+1}
\]
are in the Hurewicz image of $k\bbR$. Their preimages are the families
\[
\{\rho^\epsilon\tau^{4k+1}P^lh_1\}
\subseteq \pi_\star^{C_2}S^{0,0}
\]
which are lifts of the classical $\mu$-family.
\item the classes
\[
\frac{\theta}{\rho^{1-\epsilon}\tau^{4k+1}}\bar{v}_1^{4l}
\]
are in the Hurewicz image of $k\bbR$.
\end{enumerate}
\end{lemma}
\begin{proof}
    By \cite[Theorem~7.7]{behrensshah2020c2}, there is a $\tau^4$-self map
    \[\tau^4: \Sigma^{0, -4} S^{0,0}/a_\sigma^3 \to  S^{0,0}/a_\sigma^3,\]
    which is a homotopy equivalence. Since $\{\rho^\epsilon \tau^{4l+1} P^l h_1\}$ is $a_\sigma^3$-torsion by \eqref{eq:rhobock} and \eqref{eqn:rho_Bockstein}, it admits a factorization
    \[\begin{tikzcd}
{S^{8l+1-\epsilon, -\epsilon}} \arrow[rd, "\{\rho^\epsilon \tau^{4l+1} P^l h_1\}"] \arrow[d, "\tilde{f}"', dashed] &     \\
{\Sigma^{2, 3}S^{0,0}/a_\sigma^3} \arrow[r, "p", two heads]     & S^{0,0},
\end{tikzcd}\]
    where $p$ is the natural projection map. Since $\iota_{k\bbR}(\rho^\epsilon \tau^{4l+1} P^l h_1)=\rho^{1+\epsilon} \tau^{4l}\bar{v}_1^{4l+1}$, we have the following commutative diagram after smashing with $k\bbR$:
    \[\begin{tikzcd}
                                                                                                                                                &                                  & {\Sigma^{8l+1-\epsilon, -\epsilon}k\bbR} \arrow[rd, "\rho^{1+\epsilon}\tau^{4l}\bar{v}_1^{4l+1}"] \arrow[d, "\tilde{f}"', dashed] &       \\
{S^{8l+1-\epsilon, -\epsilon}} \arrow[rd, "\rho^\epsilon \tau^{4l+1} P^l h_1"] \arrow[d, "\tilde{f}"', dashed] \arrow[rru, "\iota_{k\bbR}"]                                          &                                  & {\Sigma^{2, 3}k\bbR/\rho^3} \arrow[r, "p"]                                                               & k\bbR. \\
{\Sigma^{2, 3}S^{0,0}/a_\sigma^3} \arrow[r, "p", two heads] \arrow[rru, "\iota_{k\bbR}"'] \arrow["\tau^{4m}"', loop, distance=2em, in=305, out=235] & S^{0,0} \arrow[rru, "\iota_{k\bbR}"] &                                                                                                          &      
\end{tikzcd}\]
\begin{enumerate}
    \item By \cref{thm: ku_R computation}, for $k\geq0$, by $\tau^4$-periodicity,
\[p\circ \tau^{4(k-l)}\circ \iota_{k\bbR}(\tilde{f})\neq 0.\]
By commutativity,
\[\iota_{k\bbR}(p\circ \tau^{4(k-l)}\circ \tilde{f})\neq 0,\]
which gives a nontrivial infinite family of elements in $\pi_\star^{C_2}S^{0,0}$
\[S^{8l+1-\epsilon, 4l-4k-\epsilon} \overset{\tilde{f}}{\longrightarrow} \Sigma^{2, 3+4l-4k}S^{0,0}/a_\sigma^3 \overset{\tau^{4(k-l)}}{\longrightarrow} \Sigma^{2, 3}S^{0,0}/a_\sigma^3 \overset{p}{\twoheadrightarrow} S^{0,0},\]
whose Hurewicz images are $\rho^{1+\epsilon}\tau^{4k} \overline{v}_1^{4l+1}$ as desired. Comparing with the genuine $C_2$-equivariant Adams spectral sequence, the only possible family is
\[
\{\rho^\epsilon\tau^{4k+1}P^lh_1\}
\subseteq \pi_\star^{C_2}S^{0,0}.
\]
    \item By \cref{thm: ku_R computation}, for $k\geq 0$, \[\rho^{1+\epsilon}\tau^{4l}\bar{v}_1^{4l+1}=\tau^{4(k+l+1)} \cdot \frac{\theta}{\rho^{1-\epsilon}\tau^{4k+1}} \bar{v}_1^{4l}\]
in $k\bbR$. It follows that
\[p\circ \tau^{-4(k+l+1)}\circ \iota_{k\bbR}(\tilde{f})\neq 0.\]
By commutativity,
\[\iota_{k\bbR}(p\circ \tau^{-4(k+l+1)}\circ \tilde{f})\neq 0,\]
which gives a nontrivial infinite family of elements in $\pi_\star^{C_2}S^{0,0}$
\[S^{8l+1-\epsilon, 4l+4k+4-\epsilon} \overset{\tilde{f}}{\longrightarrow} \Sigma^{2, 3+4l+4k+4}S^{0,0}/a_\sigma^3 \overset{\tau^{-4(k+l+1)}}{\longrightarrow} \Sigma^{2, 3}S^{0,0}/a_\sigma^3 \overset{p}{\twoheadrightarrow} S^{0,0},\]
whose Hurewicz images are $\frac{\theta}{\rho^{1-\epsilon}\tau^{4k+1}}\bar{v}_1^{4l}$ as desired. Based on the computation of $\rho$-Bockstein spectral sequences \cite{guillouisaksen2024c2}, for $l>0$, the only possible element whose Hurewicz image equals $\frac{\theta}{\rho^{1-\epsilon} \tau^{4k+1}}\bar{v}_1^{4l}$ is $\{\frac{\theta}{\rho^{2-\epsilon}\tau^{4k}}x_{8l-1}\}$, where $x_{8l-1}$ denotes the order 2 element at stem $8l-1$ of the image of $J$ in the classical 2-primary Adams spectral sequence.
\end{enumerate}
\end{proof}

\begin{lemma}\label{lem:combine(tauh1)2}
For $k, l\geq 0$,
~\begin{enumerate}
    \item the classes $\rho^{2}\tau^{4k}\bar{v}_1^{4l+2}$ are in the Hurewicz image of $k\bbR$.
    \item the classes $\frac{\theta}{\tau^{4k+1}}\bar{v}_1^{4l+1}$ are in the Hurewicz image of $k\bbR$.
\end{enumerate}
\end{lemma}
\begin{proof}
~\begin{enumerate}
    \item Since $\iota_{k\bbR}$ is a ring map and both $\rho \bar{v}_1$ and $\rho \tau^{4k}\bar{v}_1^{4l+1}$ are in the Hurewicz image by \cref{lem:kohurewicz} and \cref{lem:combinetauh1}(1), their product must also be in the Hurewicz image. In particular, the preimage of $\rho^{2}\tau^{4k}\bar{v}_1^{4l+2}$ is $\{\tau h_1\}\cdot \{\tau^{4k+1}P^lh_1\}$.
    \item Since $\iota_{k\bbR}$ is a ring map and the classes $\frac{\theta}{\rho\tau^{4k+1}}\bar{v}_1^{4l}$ are in the Hurewicz image by \cref{lem:combinetauh1}(2), their products with $\rho\bar{v}_1$ must be in the Hurewicz image.
\end{enumerate}  
\end{proof}


\begin{lemma}\label{lem:cowt0}
    For $k \geq 0$, $\epsilon=0,1$, the elements $\frac{\theta}{\rho^\epsilon\tau^{4k+1}}\bar{v}_1^{4k+3}$ are in the Hurewicz image of $k\bbR$.
\end{lemma}
\begin{proof}
    By \cite{guillouhillisaksenravenel2020cohomology}, at coweight 0, the preimage of $\frac{\theta}{\rho^\epsilon\tau^{4k+1}}\bar{v}_1^{4k+3}$ under the map 
        \[c_\star: \pi_\star^{C_2} ko_{C_2} \to \pi_\star^{C_2} k\bbR\]
    is $\frac{\eta^{4k+4}}{\rho^{4k+2+\epsilon}}$. The claim then follows from the computation of the Bredon-Landweber region of $S^{0,0}$ in \cite{guillouisaksen2020bredon}.
\end{proof}

\begin{proof}[Proof of \cref{thm:krhurewicz}]
    By \cref{lem:kohurewicz}--\cref{lem:cowt0}, it is enough to show that the other classes are not in the Hurewicz image of $k\bbR$.

    By \cite{greenleesquigley2023ranks, belmontxuzhang2024reduced}, the only torsion-free classes in $\pi_{s, w}^{C_2}S^{0,0}$ for $s>0$ are concentrated in degree $s=w$. We claim that the classes
    \[n\cdot \frac{\theta}{\tau^{2k-2}}\bar{v}_1^{2k}, \quad k\geq 1, \, n \in \bbZ\setminus \{0\}\]
    are not in the Hurewicz image. Otherwise, since $2\tau^{2k}$ is in the Hurewicz image, the product
    \[2\tau^{2k} \cdot n\cdot \frac{\theta}{\tau^{2k-2}}\bar{v}_1^{2k} =4n \cdot\bar{v}_1^{2k}\]
    would be in the Hurewicz image, which is impossible for degree reasons.

    We divide the remaining 2-torsion classes into the following three cases.

    \begin{itemize}
        \item \textbf{coweight} $\mathbf{\equiv 2} \pmod 4$. 
        By \cref{lem:combine(tauh1)2}, the preimages of $\rho^2\tau^{4k}\bar{v}_1^{4l+2}$ and $\frac{\theta}{\tau^{-4k-3}}\bar{v}_1^{4l+1}$ under the map 
        \[c_\star: \pi_\star^{C_2} ko_{C_2} \to \pi_\star^{C_2} k\bbR\]
        are $\tau^{4k+2}h_1^2 (\tau^4\beta)^l$ and $\frac{\frac{\theta}{\tau^{-4k-4}}\cdot h_1 (\tau^4\beta)^l}{\rho}$, respectively. By the computation in \cite[Section~11]{guillouhillisaksenravenel2020cohomology}, these two classes are not $\rho$-divisible in $ko_{C_2}$. It follows that $\rho\tau^{4k}\bar{v}_1^{4l+2}$ and $\frac{\theta}{\rho\tau^{-4k-3}}\bar{v}_1^{4l+1}$ cannot be in the Hurewicz image of $k\bbR$.

        \item \textbf{coweight} $\mathbf{\equiv 3} \pmod 4$. By \cite[Section~11.6]{guillouhillisaksenravenel2020cohomology} $\pi_\star^{C_2}ko_{C_2}$ is trivial at positive coweights congruent to 3 $\pmod4$. By the factorization in \eqref{eq: kotokR}, $k\bbR$ must have trivial Hurewicz images in these degrees as well.

        For the negative coweights, by \cite[Table~11]{guillouhillisaksenravenel2020cohomology}, the classes $\frac{\theta}{\tau^{4k+2}}\rho^{1+\epsilon}\bar{v}_1^{4l+3}$ are not in the image of
        \[c_\star: \pi_\star^{C_2} ko_{C_2} \to \pi_\star^{C_2} k\bbR.\]
        Therefore, these classes do not lie in the Hurewicz image.
        
        \item \textbf{coweight} $\mathbf{\equiv 0} \pmod 4$. For the negative coweights, it follows from \cite[Section~11.3]{guillouhillisaksenravenel2020cohomology} that $ko_{C_2}$ has no generators at stems congruent to $2, 3 \pmod 4$ and coweight congruent to $0 \pmod 4$. By the factorization in \eqref{eq: kotokR}, $k\bbR$ must have trivial Hurewicz image in these degrees.

        For the positive coweights, the preimage of $\rho^{1+\epsilon}\tau^{4k}\bar{v}_1^{4l}$ under \[c_\star: \pi_\star^{C_2} ko_{C_2} \to \pi_\star^{C_2} k\bbR\] 
        is $\rho^{1+\epsilon}\tau^{4(k+l)}\beta^l$, following the notation of \cite{guillouhillisaksenravenel2020cohomology}. Note that, for all $k, l\geq 0$, the classes $\rho^{1+\epsilon}\tau^{4(k+l)}\beta^l$ are torsion-free, while $\pi_\star^{C_2} S^{0,0}$ is torsion in the relevant degrees. Therefore, the classes $\rho^{1+\epsilon}\tau^{4(k+l)}\beta^l$ are not in the Hurewicz image. It follows that the classes $\rho^{1+\epsilon}\tau^{4k}\bar{v}_1^{4l}$ do not lie in the Hurewicz image of $k\bbR$ either.
    \end{itemize}
\end{proof}

In particular, \cref{thm:krhurewicz} gives the following families in the $RO(C_2)$-graded equivariant stable stems.

\begin{cor}[\cref{thm:mainmufamily}]\label{cor: four periodic families}
   ~\begin{enumerate}
    \item For $k\geq 0$, there are two families of elements in $\pi_\star^{C_2}S^{0,0}$:
    \begin{itemize}
        \item $\{\tau^{4k+1}P^l h_1\}_{l\geq 0}$ at degree $(8l+1, 4(l-k))$,
        \item $\{\frac{\theta}{\rho^{2}\tau^{4k}}x_{8l-1}\}_{l>0}$ at degree $(8l+1, 4(l+k+1))$,
    \end{itemize}
    whose underlying is the classical $\mu_{8l+1}$-family in $\pi_*^s$. 
    \item For $k, l\geq 0$, there are two families of elements in $\pi_\star^{C_2}S^{0,0}$:
    \begin{itemize}
        \item $\{\tau^{4k+2}P^l h_1^2\}$ at degree $(8l+2, 4(l-k))$,
        \item $\{\frac{\theta}{\rho\tau^{4k}}P^l h_1\}$ at degree $(8l+2, 4(l+k+1))$,
    \end{itemize}
    whose underlying is the classical $\mu_{8l+2}$-family in $\pi_*^s$.
\end{enumerate}    
\end{cor}

\section{Some \texorpdfstring{$C_2$}{C2}-equivariant \texorpdfstring{$v_1$}{v1}-periodicity} \label{sec4}

The infinite families in \cref{cor: four periodic families} exhibit periodic behavior analogous to the classical $v_1$-periodicity associated with Adams' periodicity operator $P$. In this section, we provide a construction of an equivariant lift of the periodicity operator, and use it to construct a $C_2$-equivariant family. 

\subsection{A \texorpdfstring{$C_2$}{C2}-equivariant \texorpdfstring{$v_{1,0}$}{v1}-self map} \label{sec4.1}

We first establish the existence of a $v_{1,0}$-self map of degree $(8,4)$, in the sense of \cite{balmersanders2016spectrum, behrenscarlisle2025periodic, balderramahouzhang2026self}, on the equivariant 2-cell complex 
\[S^{0,0}/2: = \Cof(S^{0, 0} \overset{2}{\longrightarrow} S^{0,0}).\] 
Here, $2$ is detected by  $h_0+\rho h_1$ in the genuine $C_2$-equivariant Adams spectral sequence. Note that previously known examples of $C_2$-equivariant $v_{1,0}$-self maps constructed in \cite{bhattacharyaguillouli2022realizations, behrenscarlisle2025periodic, balderramahouzhang2026self} are mostly on $S^{0,0}/\bfh$, where $\bfh \in \pi_{0, 0}^{C_2}$ is detected by $h_0$.

\begin{prop}\label{prop:v1selfmap}
    There is a $C_2$-equivariant stable map 
    \[v: \Sigma^{8, 4}S^{0,0}/2 \to S^{0,0}/2,\]
     which is non-nilpotent and lifts the non-equivariant self map $v_1^4$ on $S^{0}/2$ constructed by Adams \cite{adams1966on}.
\end{prop}
\begin{proof}
    Consider the element $\{h_0^3h_3+\rho^2h_1c_0\} \in \pi_{7, 4}^{C_2}$. By \cite{guillouisaksen2024c2}, 
\[\Res^{C_2}_e(\{h_0^3h_3+\rho^2h_1c_0\})=\Res^{C_2}_e(\{h_0^3h_3\})=8\sigma,\]
and
\begin{align*}
    2\cdot \{h_0^3h_3+\rho^2h_1c_0\} &= \{h_0+\rho h_1\}\cdot \{h_0^3h_3+\rho^2h_1c_0\}\\
    &=\{h_0^4h_3+ \rho^3h_1^2c_0\}\\
    &=0.
\end{align*}
Moreover, since $\Ext_{C_2}^{8, \geq 5, 4}=0$, we have
\[
    \langle 2, \{h_0^3h_3+\rho^2h_1c_0\}, 2\rangle =0. \]
    
    By \cite[Lemma~2.0.1]{balderramahouzhang2026self}, there is a map $v$ filling in
\[\begin{tikzcd}
S^{8,4} \arrow[d, hook] \arrow[rr, " \{h_0^3h_3+\rho^2h_1c_0\}"] &  & S^{1,0}                       \\
\Sigma^{8,4} S^{0,0}/2 \arrow[rr, "v"]                                &  & S^{0,0}/2 \arrow[u, two heads]
\end{tikzcd}\]
such that $v$ induces an isomorphism on $KU_{C_2}$, and thus is non-nilpotent. Moreover, applying $\Phi^e$ to the above diagram recovers the classical $v_1^4$-self map on $S^{0}/2$ by \cite{adams1966on}.
\end{proof}

\subsection{A $C_2$-equivariant Adams periodicity operator}
For any $\alpha\in \pi_{s, w}^{C_2}$ that is $2$-torsion, let $\widetilde{\alpha}$ denote a chosen lift of $\alpha$ over $S^{0, 0}/2$; in other words, $\widetilde{\alpha}$ is a lift in the following diagram
\[\begin{tikzcd}
                                              & {S^{s+1, w}} \arrow[d, " \alpha"] \arrow[ld, "\widetilde{\alpha}"', dashed] &            \\
{S^{0, 0}/2} \arrow[r, "\partial", two heads] & {S^{1,0}} \arrow[r, "2"]                                         & {S^{1, 0}}
\end{tikzcd}\]
where $\partial: S^{0,0}/2 \to S^{1,0}$ is the natural projection map. 

\begin{construction}\label[construction]{constr:P}
    For a 2-torsion class $\alpha\in \pi_{s, w}^{C_2}$ and $l\geq 0$, define $P_{C_2}^l\alpha\in \pi_{s+8l, w+4l}^{C_2}$ to be the stable class represented by the composite
\begin{equation} \label{eq:operatorP}
    P_{C_2}^l\alpha: S^{s+1+8l, w+4l} \xrightarrow{\Sigma^{8l, 4l} \widetilde{\alpha}} \Sigma^{8l, 4l}S^{0,0}/2 \xrightarrow{v^l} S^{0,0}/2 \xrightarrow{\partial} S^{1,0},
\end{equation}
where $v$ is the map constructed in \cref{prop:v1selfmap}. Note that here the construction of $P_{C_2}^l \alpha$ depends on the choice of $\widetilde\alpha$, but this will not affect \cref{prop:PC2}.
\end{construction}

The operator $P_{C_2}$ defined in \cref{constr:P} matches the notation $``P"$ used by \cite{guillouisaksen2024c2} in the genuine $C_2$-equivariant Adams spectral sequence, as illustrated in the following example.

\begin{prop}\label{prop:PC2}
    For every $l\geq 0$, the class $P_{C_2}^l\{\tau h_1\}$ is essential and is detected by $\tau P^l h_1$ in the genuine Adams spectral sequence. In particular, it has nontrivial $k\bbR$-Hurewicz image
    \[
        \iota_{k \bbR}(P_{C_2}^l\{\tau h_1\}) = \rho \overline{v}_1^{4l+1}.
    \]
\end{prop}
\begin{proof}
    Let $\alpha = \{ \tau h_1 \} \in \pi_{1,0}^{C_2}$. By \cite{guillouisaksen2024c2},
    \[\alpha\cdot 2=\alpha\cdot \{h_0\}+\alpha\cdot\{\rho h_1\}=0.\]
    
    The following commutative diagram is immediate:
    \[
        \begin{tikzcd}
            S^{8l+2, 4l} \arrow[r, "\widetilde{\alpha}"] \arrow[d, "\iota_{k\bbR} \wedge \id "] & \Sigma^{8l, 4l}S^{0,0}/2 \arrow[r, "v^{l}"] \arrow[d, "\iota_{k\bbR} \wedge \id "] & S^{0,0}/2 \arrow[r, "\partial"] \arrow[d, "\iota_{k\bbR} \wedge \id "] & S^{1,0} \arrow[d, "\iota_{k\bbR} \wedge \id "] \\
            k\bbR \wedge S^{8l+2, 4l} \quad \arrow[r, "\id \wedge \widetilde{\alpha}"] & \quad k\bbR \wedge \Sigma^{8l, 4l}S^{0,0}/2 \arrow[r, "\id \wedge v^{l}"] & k\bbR \wedge S^{0,0}/2 \arrow[r, "\id \wedge \partial"] & k\bbR \wedge S^{1,0}.
        \end{tikzcd}
    \]
    It suffices to show the composite
    \[
        \iota_{k\bbR}(P_{C_2}^l\alpha) = \iota_{k\bbR}( \partial \circ v^{l} \circ \widetilde{\alpha}) \neq 0,
    \]
    Non-equivariantly, smashing the Wood cofiber sequence
    \[
        \Sigma ko \xrightarrow{\eta} ko \xrightarrow{c} ku
    \]
    with $S^0/2$, we have a long exact sequence
    \begin{equation} \label{eq: woodexactsequence}
        \pi_1(ko/2)\xrightarrow{\eta} \pi_2 (ko/2) \xrightarrow{c_{*}} \pi_2 (ku/2) \to \pi_0(ko/2) \xrightarrow{\eta} \pi_{1}(ko/2).
    \end{equation}
    Here $\eta\in\pi_1^s$ is the non-equivariant stable Hopf map, and $c$ is the non-equivariant complexification map corresponding to \eqref{eq: kotokR}. The relevant homotopy groups are
    \[
        \pi_2(ko/2) \cong \bbZ/4, \quad \pi_1(ko/2) \cong \bbF_2 \{ \eta \}, \quad \pi_0(ko/2) \cong \bbF_2 \{ 1 \}, \quad \pi_2(ku/2) \cong \bbF_2 \{ v_1 \}.
    \]
    Thus, $\pi_0(ko/2) \xrightarrow{\eta} \pi_{1}(ko/2)$ is an isomorphism, and $c_*$ in \eqref{eq: woodexactsequence} is surjective. 
    
    Moreover, since $\Res^{C_2}_e(\alpha)=\eta$, by construction $\partial \circ \Res_e^{C_2}(\widetilde{\alpha}) = \Res_e^{C_2}(\partial \circ \widetilde{\alpha}) = \eta$. As a consequence, for the composite,
    \[
        S^2 \xrightarrow{\Res_e^{C_2}(\widetilde{\alpha})} S^0/2 \xrightarrow{\iota_{ko} \wedge \id} ko/2,
    \]
    we have $\partial_* \left [ (\iota_{ko} \wedge \id) \circ \Res_e^{C_2}(\widetilde{\alpha}) \right ] = \eta \neq 0$. Since \[\partial_*: \pi_2(ko/2) =\bbZ/4 \to \pi_1 ko = \bbZ/2\]
    is a surjection, it follows that $(\iota_{ko} \wedge \id) \circ \Res_e^{C_2}(\widetilde{\alpha}) \in \pi_2(ko/2)$ must be a generator of $\bbZ/4$, and therefore its image under $c_*$ is nonzero. Therefore,
    \[
        (\iota_{ku} \wedge \id) \circ \Res_e^{C_2}(\widetilde{\alpha}) = v_1 \neq 0.
    \]
    Since $\Res_e^{C_2}(v)$ recovers the classical $v_1^4$-self map, we have
    \begin{equation}\label{eq:proofin4}
        \Res_e^{C_2} \left [ (\iota_{k\bbR} \wedge \id) \circ v^l \circ \widetilde{\alpha} \right ] = v_1^{4l+1} \neq 0.
    \end{equation}

    It follows that $\iota_{k\bbR}(v^l \circ \widetilde{\alpha})\in \pi_{8l+2, 4l}^{C_2} k\bbR/2$ is nonzero. Consider the long exact sequence
    \[\pi_{8l+2, 4l}^{C_2} k\bbR \longrightarrow \pi_{8l+2, 4l}^{C_2} k\bbR/2 \overset{\partial_*}{\longrightarrow} \pi_{8l+1, 4l}^{C_2}k\bbR.\]
    Since $\pi_{8l+2, 4l}^{C_2} k\bbR =\bbF_2\{\rho^2 \bar{v}_1^{4l+2}\}$, and $\Res^{C_2}_e(\rho^2 \bar{v}_1^{4l+2})=0$, \eqref{eq:proofin4} and exactness implies
    \[\iota_{k\bbR}(v^l \circ \widetilde{\alpha}) \not\in \im(\pi_{8l+2, 4l}^{C_2} k\bbR \to \pi_{8l+2, 4l}^{C_2} k\bbR/2).\]
    Therefore,
     \[\partial_*(\iota_{k\bbR}(v^l\circ \widetilde{\alpha}))=\iota_{k\bbR }(\partial \circ v^l\circ \widetilde{\alpha})\] 
     must be nontrivial. For degree reasons and \cref{thm: ku_R computation},
    \[
        \iota_{k \bbR}(P_{C_2}^l\{\tau h_1\}) = \rho \overline{v}_1^{4l+1}.
    \] 

    By \cref{lem:combinetauh1}, $P_{C_2}^l\{\tau h_1\}$ is detected by $\tau P^l h_1$ in the genuine $C_2$-equivariant Adams spectral sequence.
\end{proof}

\begin{rmk}
    Using a $\tau^4$-periodic operator on $S^{0, 0}/a_\sigma^3$ as discussed in the proof of \cref{lem:combinetauh1}, the other three families in \cref{cor: four periodic families} can be explicitly constructed. The proof is similar to that of \cref{prop:PC2}.
\end{rmk}

\begingroup
\raggedright
\bibliography{refs}
\bibliographystyle{alpha}
\endgroup

\end{document}